\documentclass[reqno,12pt]{amsart}
\usepackage[active]{srcltx}
\usepackage{color}
\usepackage{amsmath,amsthm,amssymb}
\usepackage{epsfig}
\usepackage{graphicx}
\usepackage{amssymb}
\usepackage{latexsym}
\usepackage{pdfpages}
\usepackage{mathtools}
\usepackage{xfrac}
\usepackage{nicefrac}
\usepackage{enumitem}
\usepackage{ulem}

\usepackage{ifpdf}

\ifpdf
\usepackage[hidelinks]{hyperref}
\else
\fi

\usepackage[usenames,dvipsnames]{pstricks}
\usepackage{epsfig}
\usepackage{pst-grad} 
\usepackage{pst-plot} 

\usepackage{color}

\newtheorem{theorem}{Theorem}[section]
\newtheorem{lemma}[theorem]{Lemma}
\newtheorem{definition}[theorem]{Definition}
\newtheorem{proposition}[theorem]{Proposition}
\newtheorem{corollary}[theorem]{Corollary}
\newtheorem{remark}[theorem]{Remark}
\newtheorem{example}[theorem]{Example}
\newtheorem*{theorem*}{\it Theorem}

\def\R{\mathbb R}
\def\N{\mathbb N}

\numberwithin{equation}{section}

\def\1{\raisebox{2pt}{\rm{$\chi$}}}

\def\XXint#1#2#3{{\setbox0=\hbox{$#1{#2#3}{\int}$}
		\vcenter{\hbox{$#2#3$}}\kern-.5\wd0}}

\def\XXint#1#2#3{{\setbox0=\hbox{$#1{#2#3}{\int}$}
		\vcenter{\hbox{$#2#3$}}\kern-.5\wd0}}

\newcommand{\twopartdef}[4]
{
	\left\{
	\begin{array}{ll}
		#1 & #2 \\
		#3 & #4
	\end{array}
	\right.
}

\DeclareMathOperator*{\esssup}{ess\,sup}

\mathtoolsset{showonlyrefs}

\begin{document}
	
\title[Weak solutions to metric gradient flows]{\bf Weak solutions to gradient flows of functionals \\ with inhomogeneous growth in metric spaces}
	
\author[W. G\'{o}rny]{Wojciech G\'{o}rny}

\address{ W. G\'{o}rny: Faculty of Mathematics, Universit\"at Wien, Oskar-Morgerstern-Platz 1, 1090 Vienna, Austria; Faculty of Mathematics, Informatics and Mechanics, University of Warsaw, Banacha 2, 02-097 Warsaw, Poland.
\hfill\break\indent
{\tt  wojciech.gorny@univie.ac.at }
}
	
%
%
	
	
\setcounter{tocdepth}{1}

\date{\today}
	
\begin{abstract} 
We use the framework of the first-order differential structure in metric measure spaces introduced by Gigli to define a notion of weak solutions to gradient flows of convex, lower semicontinuous and coercive functionals. We prove their existence and uniqueness and show that they are also variational solutions; in particular, this is an existence result for variational solutions. Then, we apply this technique in the case of a gradient flow of a functional with inhomogeneous growth.
\end{abstract}	
\maketitle


\section{Introduction}

The study of gradient flows of convex functionals originated in the late 60s with the works of Komura \cite{Komura}, Crandall-Pazy \cite{CP} and Brezis \cite{Brezis0}. This approach was based on the theory of maximal monotone operators in Hilbert spaces; the subdifferential of a convex functional in a Hilbert space is maximal monotone and therefore generates a contraction semigroup, which can be understood as a solution of the corresponding evolution equation. Furthermore, the solutions constructed using this approach have regularising properties (see \cite{Brezis} for an overview). The theory of gradient flows has since expanded in many directions, among which we want to highlight the approach originated by Otto \cite{Otto}, where many partial differential equations can be understood as gradient flows by resorting to probability spaces endowed with the Wasserstein metric. This in turn led to the study of gradient flows in metric spaces. A classical reference is the monograph by Ambrosio, Gigli and Savar\'e \cite{AGSBook}. The starting point of this paper are the papers \cite{AGS1,AGS2} of the same authors concerning the semigroup approach to the heat flow, the $p$-Laplacian evolution equation, and their relation to the study of metric spaces with Ricci curvature bounded from below (in the sense of the $CD(K,N)$ condition given by Sturm \cite{Sturm1,Sturm2}).

Using the semigroup approach, the heat flow has been studied by Ambrosio, Gigli and Savar\'{e} in \cite{AGS2}. Assuming the $CD(K,\infty)$ condition on the metric measure space $(\mathbb{X},d,\nu)$, the authors define the heat flow as the gradient flow in $L^2(\mathbb{X},\nu)$ of the Dirichlet-Cheeger energy
\begin{equation*}
\mathsf{Ch_2}(u) = \twopartdef{\displaystyle \frac12 \int_{\mathbb{X}} |Du|^2 \, d\nu}{\mbox{if } u \in W^{1,2}(\mathbb{X},d,\nu);}{\displaystyle +\infty}{\mbox{otherwise}}
\end{equation*}
(see Section \ref{sec:preliminaries} for the exact definitions of the objects involved) and study its properties. This gives some additional information concerning the structure of the metric measure space. If the functional $\mathsf{Ch}_2$ defines a quadratic form, we say that $\mathbb{X}$ is an $RCD(K,\infty)$ space, i.e., it has Riemannian Ricci curvature bounded from below. This condition has been since utilised to study for instance the tangent spaces to a metric measure space and to prove independence of the upper gradients on the exponent. Alternatively, the gradient flow of $\mathsf{Ch}_2$ can be understood as the gradient flow of the Boltzmann entropy with respect to the Wasserstein distance in the space of probabilities; as shown by Kell in \cite{Kell}, a similar analysis applies to the $p$-Laplacian evolution equation.

In these papers, the gradient flow in $L^2(\mathbb{X},\nu)$ of the $p$-Cheeger energy
\begin{equation*}
\mathsf{Ch_p}(u) = \twopartdef{\displaystyle \frac1p \int_{\mathbb{X}} |Du|^p \, d\nu}{\mbox{if } u \in W^{1,p}(\mathbb{X},d,\nu);}{\displaystyle +\infty}{\mbox{otherwise}}
\end{equation*}
was defined in the framework of maximal monotone operator in Hilbert spaces and the corresponding $p$-Laplacian operator is defined through the subdifferential of the $p$-Cheeger energy, but without giving a characterisation of it. A characterisation of the subdifferential of the $p$-Cheeger energy in terms of a first-order differential structure due to Gigli \cite{Gig} was given by G\'orny and Maz\'on in \cite{GM2021}. In this way, the authors introduce a notion of weak solutions to the $p$-Laplacian evolution equation, where the gradient of the solution is understood in the sense of \cite{Gig}. This approach yields a pointwise condition for verifying if a function is a solution of a gradient flow in quite general settings, such as weighted Euclidean spaces and Finsler manifolds, as well as give a new technique to study the qualitative properties of solutions. Using a general Gauss-Green formula on metric measure space introduced in \cite{GM2021-2}, this approach has been applied by the same authors to study the total variation flow \cite{GM2021-3}. In particular, the authors introduced a notion of entropy solutions to the total variation flow for initial data in $L^1(\mathbb{X},\nu)$.

The main effort in this paper is to study the gradient flows which arise from general convex functionals in metric measure spaces. Throughout the whole paper, we only consider functionals of superlinear growth and assume that $p > 1$. We first introduce a notion of weak solutions for convex and lower semicontinuous functionals $\mathcal{F}: L^2(\mathbb{X},\nu) \rightarrow [0,+\infty]$ which depend only on the differential $du$ of the function $u$ in the sense of \cite{Gig}, i.e.,
\begin{equation}
\mathcal{F}(u) = \twopartdef{E(du)}{\mbox{if } u \in W^{1,p}(\mathbb{X},d,\nu);}{+ \infty}{\mbox{otherwise},}
\end{equation}
where $E$ is continuous, convex, and coercive. The whole setting is described in detail in Section \ref{sec:preliminaries}. In particular, this setting covers the case when the functional only depends on $u$ through its minimal $p$-weak upper gradient. We then prove existence and uniqueness of weak solutions and compare them to variational solutions in the sense of Lichnewsky and Temam \cite{LT}; this also serves as an existence result for variational solutions. Let us note that these results do not require doubling, Poincar\'e, or curvature assumptions on the metric measure space $(\mathbb{X},d,\nu)$. Then, assuming that the space has Riemannian Ricci curvature bounded from below so that the minimal upper gradients do not depend on the exponent, we apply these general results to give a characterisation of the subdifferential of an operator with inhomogeneous growth: the $(q,p)$-Laplace operator, where $1 \leq q < p$, and exploit it to conclude that this operator is completely accretive.

Let us now briefly describe the structure of the paper. We will work under the standard assumptions that the metric space $(\mathbb{X},d)$ is complete and separable. Furthermore, we will require that $\nu$ is a nonnegative Radon measure which is finite on bounded sets. In Section \ref{sec:preliminaries}, we recall all the notions about analysis on metric measure spaces required in this paper; more precisely, in Section \ref{subsec:sobolev}, we recall the definition of the Sobolev space $W^{1,p}(\mathbb{X},d,\nu)$ in the metric setting, in Section \ref{subsec:diffstructure} we recall the construction of the first-order differential structure on metric measure spaces introduced by Gigli \cite{Gig}, and in Section \ref{subsec:gradientflows} we recall the notions of subdifferential in convex analysis and solutions to abstract Cauchy problems in Hilbert spaces. Then, in Section \ref{sec:general} we introduce a notion of weak solutions to gradient flows of general convex, lower semicontinuous and coercive functionals and prove their existence and uniqueness, and in Section \ref{sec:inhomogeneous} we apply this general theory for the specific case of the $(q,p)$-Laplace operator.

\section{Preliminaries}\label{sec:preliminaries}

\subsection{Sobolev functions in metric measure spaces}\label{subsec:sobolev}
	
Let $(\mathbb{X}, d, \nu)$ be a metric measure space. For any $p \in [1,\infty)$, in the literature there are several possible definitions of Sobolev spaces on $\mathbb{X}$, most prominently via $p$-upper gradients, $p$-relaxed slopes, and via test plans. On complete and separable metric spaces equipped with a Borel measure which is finite on bounded sets, all these definitions agree (see \cite{AGS1} for the case $p > 1$ and \cite{ADiM} for the case $p = 1$). Throughout the whole paper, we assume that these conditions hold. We choose to work with Newtonian spaces, following the presentation in \cite{BB}, but we stress that our approach works just as well with a different choice of the definition of Sobolev spaces.

We say that a Borel function $g$ is an {\it upper gradient} of a Borel function $u: \mathbb{X} \rightarrow \R$ if for all curves $\gamma: [0,l_\gamma] \rightarrow \mathbb{X}$ we have
$$ \left\vert u(\gamma(l_\gamma)) - u(\gamma(0)) \right\vert \leq \int_\gamma g \,  ds := \int_0^{l_\gamma} g(\gamma(t))  \, \vert \dot{\gamma}(t) \vert \, dt \,  ds,$$
where $$\vert \dot{\gamma}(t) \vert:= \lim_{\tau \to 0} \frac{\gamma(t + \tau) - \gamma(t)}{\tau}$$
is the {\it metric speed} of $\gamma$.

If this inequality holds for $p$-almost every curve, i.e. the {\it $p$-modulus}
\begin{equation}
\mathrm{Mod}_p(\Gamma) = \inf_{\rho \geq 0; \,\, \int_\gamma \rho \, ds \geq 1 \, \forall \, \gamma \in \Gamma} \int_{\mathbb{X}} \rho^p \, d\nu
\end{equation}
of the family of all curves for which it fails equals zero, then we say that $g$ is a {\it $p$-weak upper gradient} of $u$.

The Sobolev-Dirichlet class $D^{1,p}(\mathbb{X})$ consists of all Borel functions $u: \mathbb{X} \rightarrow \R$ for which there exists  an upper gradient (equivalently: a $p$-weak upper gradient) which lies in $L^p(\mathbb{X},\nu)$. The Sobolev space $W^{1,p}(\mathbb{X}, d, \nu)$ is defined as
$$W^{1,p}(\mathbb{X}, d, \nu):= D^{1,p}(\mathbb{X}) \cap L^p(\mathbb{X}, \nu).$$
In the literature, this space is sometimes called the Newton-Sobolev space (or Newtonian space) and is denoted $N^{1,p}(\mathbb{X})$. The space $W^{1,p}(\mathbb{X},d,\nu)$ is endowed with the norm
\begin{equation*}
\| u \|_{W^{1,p}(\mathbb{X},d,\nu)} = \bigg( \int_{\mathbb{X}} |u|^p \, d\nu + \inf_g \int_{\mathbb{X}} g^p \, d\nu \bigg)^{1/p},
\end{equation*}
where the infimum is taken over all upper gradients of $u$. Equivalently, we may take the minimum over the set of all $p$-weak upper gradients. If the measure $\nu$ is doubling and a weak $(1,p)$-Poincar\'e inequality is satisfied (see e.g. \cite{BB} for the definitions), Lipschitz functions are dense in $W^{1,p}(\mathbb{X},d,\nu)$.

For every $u \in D^{1,p}(\mathbb{X})$, there exists a minimal $p$-weak upper gradient $|Du| \in L^p(\mathbb{X},\nu)$, i.e. we have $|Du| \leq g$ $\nu$-a.e. for all $p$-weak upper gradients $g \in L^p(\mathbb{X},\nu)$. It is unique up to a set of measure zero. In particular, we may simply plug in $|Du|$ in the infimum in the definition of the norm in $W^{1,p}(\mathbb{X},d,\nu)$. Moreover, in \cite{AGS1} (see also \cite{DiMarinoTh}) it was proved that on complete and separable metric spaces equipped with  a nonnegative Borel measure finite on bounded sets not only the various definitions of Sobolev spaces are equivalent, but also that various definitions of $|Du|$ are equivalent, including the Cheeger gradient or the minimal $p$-relaxed slope of $u$. Clearly, this identification holds up to sets of $\nu$-measure zero, since elements of the Newton-Sobolev space are defined everywhere and in the other definitions the Sobolev functions are defined $\nu$-a.e. 

In general, the minimal $p$-weak upper gradient may depend on $p$, see for instance \cite{dMS}. It was proved in \cite{GH} that if $\mathbb{X}$ has Riemannian Ricci curvature bounded from below (see \cite{AGS2}), then the minimal $p$-weak upper gradients of a given function coincide for all $p \in (1,\infty)$, and if additionally $\mathbb{X}$ is a proper space (i.e., closed and bounded sets are compact), then the above claim is valid for all $p \in [1,\infty)$. An extended discussion on this topic can be found in \cite{AdMG}.

\subsection{The differential structure}\label{subsec:diffstructure}

In the next Sections, we introduce a notion of weak solutions to a class of equations with possibly inhomogeneous growth and study their properties. Our main tool will be the linear differential structure on a metric measure space $(\mathbb{X},d,\nu)$ introduced by Gigli. We follow Gigli \cite{Gig} and Buffa-Comi-Miranda \cite{BCM} in the introduction of this first-order differential structure.

First, we will recall some basic ideas concerning the theory of $L^p$-normed modules developed in \cite{Gig}. We start with $L^\infty$-modules and their morphisms.

\begin{definition}
A Banach space $M$ is called an {\it $L^\infty$-premodule} (over $L^\infty(\mathbb{X},\nu)$) if there exists a bilinear map from $L^\infty(\mathbb{X},\nu) \times M$ to $M$ given by
\begin{equation}
(f,v) \mapsto f \cdot v,
\end{equation}
called the {\it pointwise multiplication}, with the following properties:
\begin{equation}
(fg) \cdot v = f \cdot (g \cdot v);
\end{equation}
\begin{equation}
1 \cdot v = v;
\end{equation}
\begin{equation}
\| f \cdot v \|_M \leq \| f \|_{\infty} \| v \|_M.
\end{equation}
An {\it $L^\infty$-module} is an $L^\infty$-premodule which additionally satisfies
\begin{enumerate}
\item (Locality) For every $v \in M$ and every sequence of Borel sets $\{ A_n \}_{n \in \mathbb{N}} \subset \mathfrak{B}(\mathbb{X})$
\begin{equation}
\1_{A_n} \cdot v = 0 \,\,\,\, \forall \, n \in \N \quad \Rightarrow \quad \1_{\bigcup_{n \in \N} A_n} \cdot v = 0.
\end{equation}
\item (Gluing) For any sequences $\{ v_n \}_{n \in \N} \subset M$ and $\{ A_n \}_{n \in \mathbb{N}} \subset \mathfrak{B}(\mathbb{X})$ such that
\begin{equation}
\1_{A_i \cap A_j} \cdot v_i = \1_{A_i \cap A_j} \cdot v_j \,\,\, \forall \, i,j \in \N \quad \mbox{and} \quad \limsup_{n \rightarrow \infty} \bigg\| \sum_{i = 1}^n \1_{A_i} \cdot v_i \bigg\|_M < \infty,
\end{equation}
there exists $v \in M$ such that
\begin{equation}
\1_{A_i} \cdot v = \1_{A_i} \cdot v_i \,\,\, \forall \, i \in \N \quad \mbox{and} \quad \| v \|_M \leq \liminf_{n \rightarrow \infty} \bigg\| \sum_{i = 1}^n \1_{A_i} \cdot v_i \bigg\|_M.
\end{equation}
\end{enumerate}
\end{definition}

\begin{definition}
Let $M,N$ be two $L^\infty$-modules. A bounded linear map $T: M \rightarrow N$ between Banach spaces $M$ and $N$ is a module morphism whenever
\begin{equation}
T(f \cdot v) = f \cdot T(v) \qquad \forall v \in M, \,\, f \in L^\infty(\mathbb{X},\nu).
\end{equation}
The set of all module morphisms between $M$ and $N$ will be denoted by ${\rm HOM}(M,N)$. It has a canonical structure of an $L^\infty$-module, equipped (as a Banach space) with the operator norm
\begin{equation}
\| T \| = \sup_{v \in M, \, \| v \|_M \leq 1} \| T(v) \|_N.
\end{equation}
\end{definition}

Observe that $L^1(\mathbb{X},\nu)$ has a structure of an $L^\infty$-module. Therefore, one can define a dual of a given module in the following sense.

\begin{definition}
Let $M$ be an $L^\infty$-module. The dual module $M^*$ is defined by
\begin{equation}
M^* = {\rm HOM}(M, L^1(\mathbb{X},\nu)).
\end{equation}
\end{definition}

\begin{definition}
Fix $p \in [1,\infty]$ and let $M$ be an $L^\infty$-module. We say that $M$ is an {\it $L^p$-normed module}, if there is a nonnegative map $|\cdot|_*: M \rightarrow L^p(\mathbb{X},\nu)$ such that
\begin{equation}
\| |v|_* \|_{L^p(\mathbb{X},\nu)} = \| v \|_M \quad \mbox{and} \quad | f \cdot v |_* = |f| |v|_*
\end{equation}
$\nu$-a.e. for all $f \in L^\infty(\mathbb{X},\nu)$ and all $v \in M$. We call $|\cdot|_*$ the {\it pointwise norm} on $M$.
\end{definition}

Whenever $M$ is an $L^p$-normed module, its dual module $M^*$ is an $L^{p'}$-normed module, where $\frac1p + \frac{1}{p'} = 1$; note that this holds for all $p \in [1,\infty]$. Moreover, the pointwise norm on the dual module $M^*$ is given by
\begin{equation}
|L| = \esssup \{ |L(v)|: \,\, v \in M, \,\, |v|_* \leq 1 \,\, \nu-{\rm a.e.} \}.
\end{equation}
We conclude this general introduction by recalling the notion of generating sets.

\begin{definition}
Let $M$ be an $L^\infty$-module, $A \in \mathfrak{B}(\mathbb{X})$ and $V \subset M$. The span of $V$, denoted ${\rm Span}(V)$, is the subset of $M$ which consists of $v$ such that there exist disjoint $A_n \in \mathfrak{B}(\mathbb{X})$ such that $\mathbb{X} = \bigcup_{i \in \N} A_i$ and for every elements $v_{1,n},...,v_{m_n,n} \in M$ and functions $f_{1,n}, ..., f_{m_n,n} \in L^\infty(\mathbb{X},\nu)$ such that
\begin{equation}
\1_{A_n} \cdot v = \sum_{i = 1}^{m_n} f_{i,n} \cdot v_{i,n}.
\end{equation}
We call the closure $\overline{{\rm Span}(V)}$ the {\it space generated by $V$}.
\end{definition}

We now present the construction a first-order differentiable structure on a given metric measure space $(\mathbb{X},d,\nu)$ due to Gigli \cite{Gig}, i.e., we introduce a notion of cotangent and tangent modules using the general framework described above. From now on, we assume that $\mathbb{X}$ is a complete and separable metric space and $\nu$ is a nonnegative measure which is finite on bounded sets.

\begin{definition}{\rm
The {\it pre-cotangent module} to $\mathbb{X}$ is defined as
$$ \mbox{PCM}_p = \left\{ \{(f_i, A_i)\}_{i \in \N} \ : \ (A_i)_{i \in\N} \subset \mathcal{B}(\mathbb{X}), \  f_i \in D^{1,p}(A_i), \ \ \sum_{i \in \N} \int_{A_i} |Df_i|^p \, d\nu < \infty  \right\},$$
where $A_i$ is a partition of $\mathbb{X}$. We define the equivalence relation $\sim$ as
$$\{(A_i, f_i)\}_{i \in \N} \sim \{(B_j,g_j)\}_{j \in \N} \quad \mbox{if} \quad |D(f_i - g_j)| = 0 \ \ \nu-\hbox{a.e. on} \ A_i \cap B_j.$$
The map $\vert \cdot \vert_* : \sfrac{\mathrm{PCM}_p}{\sim} \rightarrow L^p(\mathbb{X}, \nu)$, well-defined $\nu$-a.e. on $A_i$ for all $i \in \N$ and given by
$$\vert \{(f_i, A_i)\}_{i \in \N} \vert_* := \vert D f_i \vert  \quad \hbox{$\nu$-a.e. on} \ A_i  \,\, \forall \, i \in \N,$$
is the {\it pointwise norm} on $\sfrac{\mathrm{PCM}_p}{\sim}$. In $\sfrac{\mathrm{PCM}_p}{\sim}$ we define the norm $\| \cdot \|$ as
$$ \|  \{(f_i, A_i)\}_{i \in \N} \|^p = \sum_{i\in \N} \int_{A_i}|Df_i|^p, $$
and set $L^p(T^* \mathbb{X})$ to be the closure of $\sfrac{\mathrm{PCM}_p}{\sim}$ with respect to this norm, i.e. we identify functions which differ by a constant and we identify possible rearranging of the sets $A_i$.  We call $L^p(T^* \mathbb{X})$ the {\it cotangent module} and its elements will be called {\it $p$-cotangent vector fields}. It is an $L^p$-normed module.

We denote by $L^{p'}(T\mathbb{X})$ the dual module of $L^p(T^* \mathbb{X})$, namely 
$$L^{p'}(T\mathbb{X}):= \hbox{HOM}(L^p(T^* \mathbb{X}), L^1(\mathbb{X}, \nu));$$
it is a $L^{p'}$-normed module, where $\frac{1}{p} + \frac{1}{p'} = 1$. The elements of $L^{p'}(T\mathbb{X})$ will be called  {\it $p'$-vector fields} on $\mathbb{X}$. The duality between $\omega \in L^p(T^* \mathbb{X})$ and $X \in  L^{p'}(T\mathbb{X})$ will be denoted by $\omega(X) \in L^1(\mathbb{X}, \nu)$. Since the module $L^p(T^* \mathbb{X})$ is reflexive, we can identify}
$$ L^{p'}(T\mathbb{X})^* = L^p(T^*\mathbb{X}).$$
\end{definition}
	
\begin{definition}\label{dfn:differential}
{\rm
Given $f \in D^{1,p}(\mathbb{X})$, we can define its {\it differential} $df$ as an element of $L^p(T^* \mathbb{X})$ given by the formula $df = (f, \mathbb{X})$.}
\end{definition}

Clearly, the operation of taking the differential is linear as an operator from $D^{1,p}(\mathbb{X})$ to $L^p(T^{*} \mathbb{X})$, and it follows from the definition of the norm in $L^p(T^{*} \mathbb{X})$ that this operator is bounded with norm equal to one. The module $L^p(T^*\mathbb{X})$ is generated by differentials of Sobolev functions; in fact, one may equivalently define the cotangent module as the unique $L^p$-normed module generated by the differentials of Sobolev functions, see \cite[Theorem/Definition 2.8]{GigLectNotes}. As a consequence, the function
\begin{equation}
|X| = \esssup \, \{ |df(X)|: \,\, f \in D^{1,p}(\mathbb{X}), \,\, |df|_* \leq 1 \}.
\end{equation}
defines a pointwise norm on $L^{p'}(T\mathbb{X})$.

The differential satisfies the following {\it chain rule} (shown in \cite[Corollary 2.2.8]{Gig} for $p = 2$, but the proof works for all $1 < p < \infty$).

\begin{proposition}\label{prop:chainrule}
For every $f \in D^{1,p}(\mathbb{X})$ and $\varphi : \R \rightarrow \R$ Lipschitz we have
$$d( \varphi \circ  f) = (\varphi' \circ f) \, df \quad \nu-\hbox{a.e.,}$$
where $\varphi' \circ f$ is defined arbitrarily on $f^{-1}($\hbox{$\{$non differentiability points of $\varphi \})$}.
\end{proposition}

For a vector field in $L^{p'}(T\mathbb{X})$, it is also possible to define its divergence via an integration by parts formula, in the case when it can be represented by an integrable function. Let us recall the definition in the form given in \cite{GM2021} (see \cite{BCM,DiMarinoTh} for a detailed discussion of the case when the exponents coincide). For $\frac1r + \frac1s = 1$, we set
\begin{align*}
\mathcal{D}^{p',r}(\mathbb{X}) = \bigg\{ X \in L^{p'}(T\mathbb{X}): \,\, \exists f \in L^r(\mathbb{X},\nu) \quad \forall g \in &W^{1,p}(\mathbb{X},d,\nu) \cap L^s(\mathbb{X},\nu) \\
&\int_\mathbb{X} fg \, d\nu = - \int_{\mathbb{X}} dg(X) \, d\nu \bigg\}.
\end{align*}
In the above definition, the right hand side makes sense as an action of an element of $L^p(T^* \mathbb{X})$ on an element of $L^{p'}(T\mathbb{X})$; the resulting function is an element of $L^1(\mathbb{X},\nu)$. The function $f$, which is unique by the density of $W^{1,p}(\mathbb{X},d,\nu)$ in $L^{p}(\mathbb{X},\nu)$, will be called the {\it $(p',r)$-divergence} of the vector field $X$, and we shall write $\mbox{div}(X) = f$. Whenever Lipschitz functions are dense in $W^{1,p}(\mathbb{X},d,\nu)$ (e.g. as in \cite[Theorem 5.1]{BB}), the divergence does not depend on $r$ in the following sense: if $f$ is the $(p',r)$-divergence of $X$ and $f \in L^{s}(\mathbb{X},\nu)$, then it is also the $(p',s)$-divergence of $X$. However, in this paper this is of minor importance as we will only consider the case $r = 2$. A more important point concerns the dependence of the divergence on $q$. It was observed in \cite{BCM} that whenever the minimal $q$-weak upper gradient does not depend on $q$, then also the differentials and divergence do not depend on $q$. Therefore, in the part of the paper which deals with inhomogeneous growth (Section \ref{sec:inhomogeneous}) we will assume that $(\mathbb{X},d,\nu)$ is an RCD space.

The first order differential structure presented above will be our main tool in this paper. Note that a priori it is not defined locally - the objects $T^* \mathbb{X}$ and $T \mathbb{X}$ are not necessarily well-defined (the notation $L^p(T^* \mathbb{X})$ and $L^{p'}(T\mathbb{X})$ is purely formal). However, when the metric measure space is $(\mathbb{R}^N,d_{Eucl}, \mathcal{L}^N)$, the vector fields and differentials arising from this construction coincide with their standard counterparts defined in coordinates \cite{Gig}, and in a more general setting some positive results on pointwise identification of the tangent and cotangent modules can be found in \cite{EBS,LP,LPR}.

\subsection{Subdifferentials of convex functionals and gradient flows}\label{subsec:gradientflows}

Let us now briefly recall the notion of a subdifferential of a convex function defined on a Banach space $V$. To this end, we first introduce some notation for multivalued operators between Banach spaces. Then, we recall a few classical results on semigroup solutions to gradient flows in Hilbert spaces. We follow the presentation in \cite{EkelandTemam} and \cite{Brezis} respectively. 

A multivalued operator $A$ between two Banach spaces $V$ and $W$ can be described in two equivalent ways: as a mapping $A: V \rightarrow 2^W$ or equivalently one may identify $A$ with its graph, i.e., for every $v \in V$ we set
$$ A(v) : = \left\{ w \in W: \,  (v,w) \in A \right\}.$$
We denote the domain of $A$ by
$$ D(A) : = \left\{ v \in V: \, A(v) \not= \emptyset \right\}$$
and the range of $A$ by
$$R(A) : = \bigcup_{v \in D(A)} Av.$$
A classical example of a multivalued operator is the subdifferential of a convex function, which is shortly described below.

Given a Banach space $V$ and a functional $\mathcal{F}: V \to (-\infty, \infty]$, we call the set
$$D(\mathcal{F}) : = \{ v \in V: \ \mathcal{F}(v) < + \infty \}$$
the effective domain of $\mathcal{F}$. The functional $\mathcal{F}$ is said to be proper if $D(\mathcal{F})$ is nonempty.
Furthermore, we say that $\mathcal{F}$ is lower semicontinuous if for every $c \in \R$ the sublevel set $\{ v \in D(\mathcal{F}): \ \mathcal{F}(v) \leq c \}$ is closed in $V$.

Given a proper and convex functional $\mathcal{F} : V \to (-\infty, \infty]$, its {\it subdifferential} is the set
\begin{equation}
\partial \mathcal{F} := \left\{(v_0,v^*) \in V \times V^*: \ \mathcal{F}(v_0 + v) - \mathcal{F}(v_0) \geq \langle v, v^* \rangle_{(V,V^*)} \ \ \forall \, v \in V \right\}.
\end{equation}
It is a generalisation of the derivative; in the case when $\mathcal{F}$ is Fr\'{e}chet differentiable, its subdifferential is single-valued and equals the Fr\'{e}chet derivative. Since we identify a multivalued operator with its graph, given $v_0 \in V$ we may equivalently write that
\begin{equation}
\partial \mathcal{F}(v_0) := \left\{ v^* \in V^*: \, \mathcal{F}(v_0 + v) - \mathcal{F}(v_0) \geq \langle v, v^* \rangle_{(V,V^*)} \ \ \forall \, v \in V \right\}.
\end{equation}
One of the classical tools used to characterise the subdifferential of a convex functional $\mathcal{F}$ is convex duality (a standard reference is \cite[Chapter III.4]{EkelandTemam}). First, recall the notion of the Legendre-Fenchel transform: given a Banach space $V$ and $F: V \rightarrow \mathbb{R} \cup \{ + \infty \}$, we define $F^*: V^* \rightarrow \mathbb{R} \cup \{ + \infty \}$ by the formula
\begin{equation*}
F^*(v^*) = \sup_{v \in V} \bigg\{ \langle v, v^* \rangle - F(v) \bigg\}.
\end{equation*}
Then, let us briefly describe how the dual problem is typically defined in the setting of calculus of variations. 

Let $U, V$ be two Banach spaces and let $A: U \rightarrow V$ be a continuous linear operator. Denote by $A^*: V^* \rightarrow U^*$ its dual. Then, if the primal problem is of the form
\begin{equation}\tag{P}\label{eq:primal}
\inf_{u \in U} \bigg\{ E(Au) + G(u) \bigg\},
\end{equation}
where $E: V \rightarrow (-\infty,+\infty]$ and $G: U \rightarrow (-\infty,+\infty]$ are proper, convex and lower semicontinuous, then the dual problem is defined as the maximisation problem
\begin{equation}\tag{P*}\label{eq:dual}
\sup_{v^* \in V^*} \bigg\{ - E^*(-v^*) - G^*(A^* v^*) \bigg\}
\end{equation}
It turns out that in the above setting the minimum value in the primal problem equals the maximum value in the dual problem; this is the content of the Fenchel-Rockafellar duality theorem, which we recall in the form given in \cite[Remark III.4.2]{EkelandTemam}. 

\begin{theorem}\label{thm:fenchelrockafellartheorem}
Suppose that $E$ and $G$ are proper, convex and lower semicontinuous. If there exists $u_0 \in U$ such that $E(A u_0) < \infty$, $G(u_0) < \infty$ and $E$ is continuous at $A u_0$, then
$$\inf \eqref{eq:primal} = \sup \eqref{eq:dual}$$
and the dual problem \eqref{eq:dual} has at least one solution. Moreover, if $u$ is a solution of \eqref{eq:primal} and $v^*$ is a solution of \eqref{eq:dual}, the following extremality conditions hold:
\begin{equation}
E(Au) + E^*(-v^*) = \langle -v^*, Au \rangle;
\end{equation}
\begin{equation}
G(u) + G^*(A^* v^*) = \langle u, A^* v^* \rangle.
\end{equation}
\end{theorem}

Let us turn our attention to the gradient flows of convex functionals in Hilbert spaces. We say that a multivalued operator $A$ on a Hilbert space $H$ is {\it monotone} if
$$\langle u - \hat{u}, v - \hat{v} \rangle_H \geq 0$$
for all $(u,v), (\hat{u}, \hat{v}) \in A.$
If there is no monotone operator which strictly contains $A$, we say that $A$ is maximal monotone. The prime example of a multivalued operator is the subdifferential; if $\mathcal{F} : H \to (-\infty, \infty]$ is convex and lower semicontinuous, then $\partial \mathcal{F}$ is a maximal monotone multivalued operator on $H$. In general, the following theorem due to Minty provides a characterisation of maximal monotone operators.

\begin{theorem}\label{thm:minty}
Let $H$ be a Hilbert space. A monotone multivalued operator $A: H \rightarrow 2^H$ is maximal monotone if and only if $R(I + A) = H$. 
\end{theorem}

The natural function spaces for solutions to gradient flows of convex functionals in Hilbert spaces are the following. For $1 \leq p < \infty$, denote
$$L^p(a,b; H):= \left\{ u : [a,b] \rightarrow H \ \hbox{measurable such that} \  \int_a^b \Vert u(t) \Vert_H^p \, dt < \infty \right\}$$
and
\begin{align}
W^{1,p}(a,b;H):= &\bigg \{ u \in L^p(a,b; H) \ \hbox{and} \ \exists \, v \in L^p(a,b; H): \\
&\qquad\qquad\qquad\qquad u(t) - u(a) = \int_a^t v(s) \, ds \ \ \forall \, t \in (a,b)  \bigg \}.
\end{align}
If $u \in W^{1,p}(a,b;H)$, it is differentiable for almost all $t \in (a,b)$ and
$$u(t) - u(a) = \int_a^t \frac{du}{dt}(s) \, ds \ \ \forall \, t \in (a,b).$$
We also set $W_{\rm loc}^{1,p}(0,T;H)$ to be the space of all functions $u$ with the following property: for all $0 < a < b < T$, we have that $u \in W^{1,p}(a,b;H)$.

Consider the abstract Cauchy problem
\begin{equation}\label{eq:abstractcauchy}
\left\{ \begin{array}{ll} 0 \in \frac{du}{dt} + \partial \mathcal{F} (u(t)) \, \quad &\mbox{for } t \in (0, T);  \\
[10pt] u(0) = u_0, \quad & \end{array} \right.
\end{equation}
where $u_0 \in H$.

\begin{definition} 
We say that $u \in C([0,T]; H)$ is a strong solution to problem \eqref{eq:abstractcauchy}, if the following conditions hold: $u \in W_{\rm loc}^{1,2}(0,T;H)$; for almost all $t \in (0,T)$ we have $u(t) \in D(\partial \mathcal{F})$; and it satisfies \eqref{eq:abstractcauchy}.
\end{definition}

The following result, called the Brezis-Komura theorem, summarises the existence theory for solutions to the abstract Cauchy problem \eqref{eq:abstractcauchy} obtained using the semigroup approach. In this paper, we will only apply this result with $H = L^2(\mathbb{X},\nu)$.

\begin{theorem}\label{thm:breziskomura}
Let $\mathcal{F} : H \to (-\infty, \infty]$ be a proper, convex, and lower semicontinuous functional. Given $u_0 \in \overline{D(\mathcal{F})}$, there exists a unique strong solution to the abstract Cauchy problem \eqref{eq:abstractcauchy}. If additionally $u_0 \in D(\mathcal{F})$, then $u \in W^{1,2}(0, T; H)$. 
\end{theorem} 

We refer to \cite[Theorem 3.2]{Brezis} for a summary of main additional properties of solutions; let us only briefly mention the semigroup property, the $T$-contraction property, and the regularity of time derivative. If we denote by $S(t) u_0$ the unique strong solution $u(t)$ of the abstract Cauchy problem \eqref{eq:abstractcauchy} for initial data $u_0$, then $S(t) : \overline{D(\mathcal{F})} \rightarrow H$ is a continuous semigroup satisfying the $T$-contraction property
$$\Vert (S(t) u_0 - S(t) v_0) \Vert_H \leq \Vert u_0 -  v_0 \Vert_H$$
for all $u_0, v_0 \in \overline{D(\mathcal{F})}$ and $t > 0$. Furthermore, we have that $\frac{du}{dt} \in L^2_{\rm loc}(0, T; H)$, and the function $t \mapsto \mathcal{F}(u(t))$ is convex, decreasing, and locally Lipschitz with the derivative (defined for a.e. $t > 0$)
\begin{equation}
\frac{d}{dt} \mathcal{F}(u(t)) = - \bigg\| \frac{du}{dt}(t) \bigg\|_H^2 = - \bigg\| \partial^- \mathcal{F}(u(t)) \bigg\|_H^2,
\end{equation}
where $\partial^- \mathcal{F}(\cdot)$ denotes the element of minimal norm in $\partial\mathcal{F}(\cdot)$. In fact, $\frac{du}{dt}(t) = \partial^- \mathcal{F}(u(t))$.

We conclude this subsection by the following result given in \cite[Proposition 2.11]{Brezis}, which is useful to characterise the closure of the domain of the subdifferential.

\begin{proposition}\label{prop:densityofthedomain} 
Let $\mathcal{F}: H \to (-\infty, \infty]$ be a proper, convex, and lower semicontinuous functional. Then,
$$D(\partial \mathcal{F}) \subset D( \mathcal{F}) \subset \overline{ D( \mathcal{F})} \subset \overline{ D( \partial \mathcal{F})}.$$
\end{proposition}

\section{Characterisation of the subdifferential in the general case}\label{sec:general}

In this Section, we introduce a notion of weak solutions to gradient flows in metric measure spaces in a fairly general setting, assuming that the functional only depends on the differential of the function $u$. In particular, this setting covers the case when the functional only depends on $u$ through its minimal $p$-weak upper gradient, since in this case we have $|du|_* = |Du|$ $\nu$-a.e. As usual, we require some form of convexity, lower semicontinuity, and coercivity of the functional; let us briefly describe below the setting for this entire Section.

{\bf \flushleft Assumptions.} Let $p > 1$ and suppose that $\mathcal{F}: L^2(\mathbb{X},\nu) \rightarrow [0,+\infty]$ is a convex and lower semicontinuous functional given by the formula
\begin{equation}
\mathcal{F}(u) = \twopartdef{E(du)}{\mbox{if } u \in W^{1,p}(\mathbb{X},d,\nu);}{+ \infty}{\mbox{otherwise},}
\end{equation}
where $E: L^p(T^* \mathbb{X}) \rightarrow \mathbb{R}$ is a continuous and convex functional, which satisfies the following coercivity condition: there exists $C > 0$ such that for all $v \in L^p(T^* \mathbb{X})$
\begin{equation}\label{eq:generalboundfrombelow}
E(v) \geq C \int_{\mathbb{X}} |v|_*^p.
\end{equation}
In particular, it is nonnegative. 

Then, we can define weak solutions in terms of the following multivalued operator.

\begin{definition}
We say that $(u,v) \in \mathcal{A}$ if and only if $u,v \in L^2(\mathbb{X},\nu)$, $u \in W^{1,p}(\mathbb{X},d,\nu)$ and there exists a vector field $X \in \mathcal{D}^{p',2}(\mathbb{X})$ such that the following conditions hold:
\begin{equation}
-\mathrm{div}(X) = v \quad \mbox{in } \mathbb{X};
\end{equation}
\begin{equation}
X \in \partial E(du).
\end{equation}
\end{definition}

\begin{theorem}\label{thm:generalresult}
We have $\mathcal{A} = \partial \mathcal{F}$. Furthermore, its domain is dense in $L^2(\mathbb{X},\nu)$.
\end{theorem}

\begin{proof}
{\bf Step 1.} First, let us see that $\mathcal{A} \subset \partial \mathcal{F}$. Let $(u,v) \in \mathcal{A}$. Given $w \in W^{1,p}(\mathbb{X}, d, \nu)$, we calculate
\begin{align}
\int_{\mathbb{X}} v&(w-u) \, d \nu = - \int_{\mathbb{X}}  \mathrm{div}(X)(w - u) \, d\nu = \int_{\mathbb{X}} d(w-u)(X) \, d\nu \\
&= \int_{\mathbb{X}} dw(X) \, d\nu - \int_{\mathbb{X}} du(X) \, d\nu \leq E(dw) - E(du) = \mathcal{F}(w) - \mathcal{F}(u).
\end{align}
Whenever $w \notin W^{1,p}(\mathbb{X},d,\nu)$, then $\mathcal{F}(w) = +\infty$ and this inequality also holds. Consequently, $(u,v) \in \partial \mathcal{F}$. Notice that since $\mathcal{A} \subset \partial \mathcal{F}$, the operator $\mathcal{A}$ is monotone.
{\flushleft \bf Step 2.} Since $\mathcal{F}$ is convex and lower semicontinuous, the operator $\partial \mathcal{F}$ is maximal monotone. So, if we show that  $\mathcal{A}$ satisfies the range condition, by the Minty theorem (Theorem \ref{thm:minty}) we would also have that the operator $\mathcal{A}$ is maximal monotone, and consequently $\partial \mathcal{F} = \mathcal{A}$. In order to finish the proof, we need to show $\mathcal{A}$ satisfies the range condition, i.e.
\begin{equation}
\hbox{Given} \ g \in L^2(\mathbb{X}, \nu) \ \exists \, u \in D(\mathcal{A}) \ \hbox{such that} \ \  g \in u + \mathcal{A}(u).
\end{equation}
We rewrite it as
$$ g \in u + \mathcal{A}(u) \iff (u, g-u) \in \mathcal{A},$$
so we need to show that there exists a vector field $X \in \mathcal{D}^{p',2}(\mathbb{X})$ such that the following conditions hold:
\begin{equation}
-\mathrm{div}(X) = g-u \quad \hbox{in } \mathbb{X};
\end{equation}
\begin{equation}
X \in \partial E(du).
\end{equation}
We will prove the claim using the Fenchel-Rockafellar duality theorem (Theorem \ref{thm:fenchelrockafellartheorem}). Set
$$U = W^{1,p}(\mathbb{X},d,\nu) \cap L^2(\mathbb{X},\nu); \qquad V = L^p(T^{*} \mathbb{X})$$
and define the operator $A: U \rightarrow V$ by the formula
$$A(u) = du,$$
where $du$ is the differential of $u$ in the sense of Definition \ref{dfn:differential}. Hence, $A$ is a linear and continuous operator. Moreover, the dual spaces to $U$ and $V$ are
$$ U^* = (W^{1,p}(\mathbb{X},d,\nu) \cap L^2(\mathbb{X},\nu))^*; \qquad V^* =  L^{p'}(T\mathbb{X}).$$
The functional $E: L^p(T^{*} \mathbb{X}) \rightarrow \mathbb{R}$ is the one from the definition of $\mathcal{F}$. We also set $G: W^{1,p}(\mathbb{X},d,\nu) \cap L^2(\mathbb{X},\nu) \rightarrow \mathbb{R}$ by
$$G(u):= \frac12 \int_{\mathbb{X}} u^2 \, d\nu - \int_{\mathbb{X}} ug \, d\nu,$$
so $G$ is continuous and convex. By the Young inequality we have
\begin{equation}
G(u) \geq \frac12 \int_{\mathbb{X}} u^2 \, d\nu - \varepsilon \int_{\mathbb{X}} u^2 \, d\nu - C(\varepsilon) \int_{\mathbb{X}} g^2 \, d\nu,
\end{equation}
so if we choose $\varepsilon < \frac12$, we get that $G$ is bounded from below.
 
{\flushleft \bf Step 3.} We now observe that for any $v^* \in L^{p'}(T\mathbb{X})$ in the domain of $A^*$, whenever $u \in W^{1,p}(\mathbb{X},d,\nu) \cap L^2(\mathbb{X},\nu)$ we have
$$\int_{\mathbb{X}} u \, (A^* v^*) \, d\nu = \langle u, A^* v^* \rangle  = \langle v^*, Au \rangle = \int_{\mathbb{X}} du(v^*) \, d\nu,$$
so the definition of the divergence of $v^*$ is satisfied with
\begin{equation}
\mathrm{div}(v^*) = - A^* v^*.
\end{equation}
In particular, $\mathrm{div}(v^*) \in L^2(\mathbb{X},\nu)$. In other words, the domain of $A^*$ is $\mathcal{D}^{p',2}(\mathbb{X}).$

{\flushleft \bf Step 4.} Consider the energy functional $\mathcal{G}: L^2(\mathbb{X}, \nu) \rightarrow (-\infty, + \infty]$ defined by
\begin{equation}
\mathcal{G}(u) := \mathcal{F}(u) + G(u).
\end{equation}
Clearly, $\mathcal{G}$ is convex and lower semicontinuous. It is also coercive, because whenever $\mathcal{G}(u) \leq M$, we have
\begin{equation}
\frac12 \int_{\mathbb{X}} u^2 \, d\nu + \mathcal{F}(u) \leq M + \int_{\mathbb{X}} u g \, d\nu,
\end{equation}
and by nonnegativity of $\mathcal{F}$ and the Young inequality for $\varepsilon < \frac12$ we get
\begin{equation}
\bigg(\frac12 - \varepsilon \bigg) \int_{\mathbb{X}} u^2 \, d\nu \leq M + C(\varepsilon) \int_{\mathbb{X}} g^2 \, d\nu,
\end{equation}
so the norm of $u$ in $L^2(\mathbb{X},\nu)$ is bounded. Therefore, the minimization problem
$$\min_{u \in L^2(\mathbb{X}, \nu)} \mathcal{G}(u)$$
admits an optimal solution $\overline{u} \in L^2(\mathbb{X},\nu)$; since $E$ satisfies the bound \eqref{eq:generalboundfrombelow}, we have that $\overline{u} \in W^{1,p}(\mathbb{X},d,\nu)$. Notice that we may write the minimisation problem as
\begin{equation}\label{eq:generalprimal}
\min_{u \in L^2(\mathbb{X}, \nu)} \mathcal{G}(u) = \inf_{u \in W^{1,p}(\mathbb{X},d,\nu) \cap L^2(\mathbb{X},\nu)} \bigg\{ E(Au) + G(u) \bigg\}.
\end{equation}
Hence, its dual problem is
\begin{equation}\label{eq:generaldual}
\sup_{v^* \in L^{p'}(T\mathbb{X})} \bigg\{  - E^*(-v^*) - G^*(A^* v^*) \bigg\}.
\end{equation}
For $u_0 \equiv 0$ we have $E(Au_0) = 0 < \infty$, $G(u_0) = 0 < \infty$ and $E$ is continuous at $0$. By the Fenchel-Rockafellar duality theorem (Theorem \ref{thm:fenchelrockafellartheorem}), there is no duality gap (i.e. $\inf \eqref{eq:generalprimal} = \sup \eqref{eq:generaldual})$ and the dual problem \eqref{eq:generaldual} admits at least one solution $\overline{v}$. Now, since both the primal and dual problem admit solutions, we now use the extremality conditions between the solutions $\overline{u}$ of the primal problem and $\overline{v}$ of the dual problem given in Theorem \ref{thm:fenchelrockafellartheorem}; these are
\begin{equation}
E(A\overline{u}) + E^*(-\overline{v}^*) = \langle -\overline{v}^*, A\overline{u} \rangle
\end{equation}
and
\begin{equation}
G(\overline{u}) + G^*(A^* \overline{v}^*) = \langle \overline{u}, A^* \overline{v}^* \rangle.
\end{equation}
We rewrite them respectively as 
\begin{equation}\label{eq:generalfirstcondition}
A^* \overline{v}^* \in \partial G(\overline{u})
\end{equation}
and
\begin{equation}\label{eq:generalsecondcondition}
-\overline{v}^* \in \partial E(A \overline{u}).
\end{equation}
Since $\partial G(\overline{u}) = \{ \overline{u} - g \}$, the condition \eqref{eq:generalfirstcondition} implies
\begin{equation}
-\mathrm{div}(\overline{v}^*) = \overline{u} - g,
\end{equation}
and condition \eqref{eq:generalsecondcondition} implies the range condition for $X = -\overline{v}^*$.

Hence, the operator $\mathcal{A}$ satisfies the range condition, so by the Minty theorem (Theorem \ref{thm:minty}) it is maximal monotone; therefore, $\mathcal{A} = \partial \mathcal{F}$. By Proposition \ref{prop:densityofthedomain}, we have
$$ D(\partial \mathcal{F}) \subset  D(\mathcal{F}) =  W^{1,p}(\mathbb{X},d,\nu) \cap L^2(\mathbb{X},\nu) \subset \overline{D(\mathcal{F})}^{L^2(\mathbb{X}, \nu)} \subset \overline{D(\partial \mathcal{F})}^{L^2(\mathbb{X}, \nu)},$$
from which follows the density of the domain.
\end{proof}

The above Theorem enables us to give a notion of weak solutions to the abstract Cauchy problem associated to the gradient flow of the functional $\mathcal{F}$, i.e.
\begin{equation}\label{eq:generalevolutionequation}\left\{ \begin{array}{ll} - \frac{d}{dt} u(t) \in \partial \mathcal{F}(u(t)), \quad t \in [0,T] \\[5pt] u(0) = u_0. \end{array}\right.
\end{equation}

\begin{definition}\label{dfn:generalweaksolutions}
{\rm  Given $u_0 \in L^2(\mathbb{X},\nu)$, we say that $u$ is a {\it weak solution} of the Cauchy problem \eqref{eq:generalevolutionequation}, if $u \in C([0,T];L^2(\mathbb{X},\nu)) \cap W_{\rm loc}^{1,2}(0, T; L^2(\mathbb{X},\nu))$, $u(0, \cdot) = u_0$, and for almost all $t \in (0,T)$
\begin{equation}
-u_t(t, \cdot) \in \partial\mathcal{F}(u(t, \cdot)).
\end{equation}
In other words, if $u(t) \in W^{1,p}(\mathbb{X}, d, \nu)$ and there exist vector fields  $X(t) \in  \mathcal{D}^{p',2}(\mathbb{X})$  such that for almost all $t \in [0,T]$ the following conditions hold:
$$ u_t(t, \cdot) = \mathrm{div}(X(t)) \quad \hbox{in } \mathbb{X}; $$
\begin{equation*}
X(t) \in \partial E(du(t)).
\end{equation*}
}
\end{definition}

As a consequence of the characterisation given in Theorem \ref{thm:generalresult} and the Brezis-Komura theorem (Theorem \ref{thm:breziskomura}), we have the following existence and uniqueness result.

\begin{theorem}\label{thm:generalexistence}
For any $u_0 \in L^2(\mathbb{X}, \nu)$ and all $T > 0$ there exists a unique weak solution $u(t)$ of the Cauchy problem \eqref{eq:generalevolutionequation} with $u(0) =u_0$.
\end{theorem}

By construction, the solution automatically enjoys all the properties listed below the statement of Theorem \ref{thm:breziskomura} (for a further discussion we refer to \cite{Brezis}). As a first application of the Theorem, let us now see that the notion of solutions given here agrees with the one given for the $p$-Laplacian evolution equation in \cite{GM2021}.

\begin{example}\label{ex:plaplace}
Observe that we can write the Cheeger energy $\mathsf{Ch}_p: L^2(\mathbb{X},\nu) \rightarrow [0,+\infty]$, which is given by the formula
\begin{equation}
\mathsf{Ch}_p(u) = \twopartdef{\displaystyle \frac1p \int_{\mathbb{X}} |Du|^p \, d\nu}{\mbox{if } u \in W^{1,p}(\mathbb{X},d,\nu);}{\displaystyle + \infty}{\mbox{otherwise,}}
\end{equation}
in the following way
\begin{equation}
\mathsf{Ch}_p(u) = \twopartdef{E_p(du)}{\mbox{if } u \in W^{1,p}(\mathbb{X},d,\nu);}{+\infty}{\mbox{otherwise,}}
\end{equation}
where $E_p: L^p(T^* \mathbb{X}) \rightarrow \mathbb{R}$ given by
\begin{equation}
E_p(v) = \frac1p \int_{\mathbb{X}} |v|_*^p \, d\nu
\end{equation}
is convex, continuous, nonnegative and coercive. Therefore, we may apply Theorem \ref{thm:generalexistence} to obtain existence of weak solutions in the sense of Definition \ref{dfn:generalweaksolutions}. However, in this case we may compute the subdifferential of $E_p$ in the following way: first, observe that its convex conjugate $E_p^*: L^{p'}(T\mathbb{X}) \rightarrow \mathbb{R}$ is
\begin{equation}
E_p^*(v^*) = \frac{1}{p'} |v^*|^{p'} \, d\nu.
\end{equation}
Since the condition $X \in \partial E_p(v)$ is equivalent to
\begin{equation}
E_p(v) + E_p^*(X) = \int_{\mathbb{X}} v(X) \, d\nu,
\end{equation} 
or equivalently
\begin{equation*}
\frac1p \int_\mathbb{X} |v|_*^p \, d\nu + \frac{1}{p'} \int_\mathbb{X} |X|^{p'} \, d\nu = \int_{\mathbb{X}} v(X) \, d\nu,
\end{equation*}
we get that $X \in \partial E_p(v)$ if and only if $v(X) = |X|^{p'} = |v|_*^p$ $\nu$-a.e. in $\mathbb{X}$. Therefore, the operator $\mathcal{A}$ coincides with the operator $\mathcal{A}_p$ given in \cite{GM2021}. Moreover, the definition of the weak solution, i.e. that $u$ is sufficiently regular and there exists a vector field $X \in \mathcal{D}^{p',2}(\mathbb{X})$ such that
\begin{equation}
u_t(t) = \mathrm{div}(X(t))
\end{equation}
and
\begin{equation}
du(t)(X(t)) = |X(t)|^{p'} = |du|_*^p \quad \nu-\mbox{a.e. in } \mathbb{X}
\end{equation}
coincides with the one given in \cite{GM2021}.
\end{example}

On a final note, we observe that the notion of weak solutions introduced in this Section agrees with the notion of variational solutions, which goes back to the work of Lichnewsky and Temam \cite{LT} and was formally introduced by B\"ogelein, Duzaar and Marcellini \cite{BDM} to study evolution equations with inhomogeneous growth.

\begin{proposition}
Let $u$ be a weak solution to the Cauchy problem \eqref{eq:generalevolutionequation}. Then, for any $v \in L^1(0,T; W^{1,p}(\mathbb{X},d,\nu)) \cap C([0,T]; L^2(\mathbb{X},\nu))$ with $\partial_t v \in L^2(\mathbb{X} \times [0,T], \nu \otimes \mathcal{L}^1)$ we have
\begin{align}
\int_0^T \int_{\mathbb{X}} \partial_t v(v-u) \, d\nu \, dt + \int_0^T &\mathcal{F}(v(t)) \, dt - \int_0^T \mathcal{F}(u(t)) \, dt \label{eq:variationalsolution} \\
&\geq \frac12 \Vert (v-u)(T) \Vert^2_{L^2(\mathbb{X},\nu)} - \frac12 \Vert v(0)-u_0 \Vert^2_{L^2(\mathbb{X},\nu)},
\end{align}
i.e., $u$ is a variational solution to the Cauchy problem \eqref{eq:generalevolutionequation}.
\end{proposition}

\begin{proof}
Given a test function $v$ as in the statement, we need to show that \eqref{eq:variationalsolution} holds. By taking as $v$ a function constant in time which lies in $D(\mathcal{F})$, we see that
\begin{equation}
\int_0^T \mathcal{F}(u(t)) \, dt < \infty.
\end{equation}
Since $\partial_t u \in L^2_{\rm loc}(0,T; L^2(\mathbb{X},\nu))$, let us fix some $\delta > 0$ and compute the term with the time derivative restricted to the interval $(\delta,T)$ using the characterisation of weak solutions. By the definition of the divergence, we have
\begin{align*}
\int_\delta^T \int_{\mathbb{X}} \partial_t u \, &(v - u) \, d \nu \, dt = \int_\delta^T \int_{\mathbb{X}} \mbox{div}(X(t)) (v - u) \, d \nu \, dt = \int_\delta^T \int_{\mathbb{X}} d(u - v)(X(t)) \, d\nu \, dt.
\end{align*}
Also,
$$\int_\delta^T \int_{\mathbb{X}} (\partial_t v - \partial_t u )(v - u) \, d \nu \, dt = \frac12 \Vert (v-u)(T) \Vert^2_{L^2(\mathbb{X},\nu)} - \frac12 \Vert v(\delta)- u(\delta) \Vert^2_{L^2(\mathbb{X},\nu)}. $$
We now add the two above equalities and get
\begin{align*}
\int_\delta^T &\int_{\mathbb{X}} \partial_t v \, (v - u) \, d \nu \, dt \\ &= \int_\delta^T \int_{\mathbb{X}} d(u - v)(X(t)) \, d\nu \, dt + \frac12 \Vert (v-u)(T) \Vert^2_{L^2(\mathbb{X},\nu)} - \frac12 \Vert v(\delta) - u(\delta) \Vert^2_{L^2(\mathbb{X},\nu)} \\
&\geq \int_\delta^T E(du(t)) \, dt - \int_\delta^T E(dv(t)) \, dt + \frac12 \Vert (v-u)(T) \Vert^2_{L^2(\mathbb{X},\nu)} - \frac12 \Vert v(\delta) - u(\delta) \Vert^2_{L^2(\mathbb{X},\nu)} \\
& = \int_\delta^T \mathcal{F}(u(t)) \, dt - \int_\delta^T \mathcal{F}(v(t)) \, dt + \frac12 \Vert (v-u)(T) \Vert^2_{L^2(\mathbb{X},\nu)} - \frac12 \Vert v(\delta) - u(\delta) \Vert^2_{L^2(\mathbb{X},\nu)}.
\end{align*}
The inequality follows from the condition $X(t) \in \partial E(du(t))$ in definition of weak solutions. We now pass to the limit $\delta \rightarrow 0$, and since $\partial_t v$ is square-integrable and the $L^2$ norms of the functions $u,v$ are continuous in time, we recover equation \eqref{eq:variationalsolution}.
\end{proof}

Typically, in the definition of variational solutions one assumes that $u_0 \in L^2(\mathbb{X},\nu) \cap W^{1,p}(\mathbb{X},d,\nu)$, or more generally that $u_0 \in D(\mathcal{F})$, so that by the Brezis-Komura theorem we have that $\partial_t u \in L^2(\mathbb{X} \times [0,T], \nu \otimes \mathcal{L}^1)$ and the proof is a bit simpler as we may perform the integration outright on the whole interval $[0,T]$.

\begin{remark}
In this section, we assumed that the functional $\mathcal{F}$ is lower semicontinuous with respect to convergence in $L^2(\mathbb{X},\nu)$. In general, lower semicontinuity of integral functionals in metric measure spaces is not yet completely understood; the most classical example of a lower semicontinuous functional is the Cheeger energy \cite{AGS1}. Some positive results for more general functionals can be found in \cite{HM1,HM2} and \cite{HKLL}.    
\end{remark}

\section{Application to a functional with inhomogeneous growth}\label{sec:inhomogeneous}

In this Section, we show how to apply the methods developed in the previous Section to a functional with inhomogeneous growth. Given $1 \leq q < p$, consider a functional $\mathcal{F}_{(q,p)}: L^2(\mathbb{X},\nu) \rightarrow [0, +\infty]$ given by the formula
\begin{equation}
\mathcal{F}_{(q,p)}(u) = \twopartdef{\displaystyle \frac1q \int_{\mathbb{X}} |Du|^q \, d\nu + \frac1p \int_{\mathbb{X}} |Du|^p \, d\nu}{\mbox{if } u \in W^{1,p}(\mathbb{X},d,\nu);}{\displaystyle +\infty}{\mbox{otherwise.}}
\end{equation}
We will characterize the subdifferential of $\mathcal{F}_{(q,p)}$ and apply this to define a notion of weak solutions to the gradient flow of said functional. To this end, let us first prove some technical results concerning the relation between cotangent and tangent modules for different exponents; since the functional $\mathcal{F}_{(q,p)}$ has two terms with different growth, until the end of the paper we will work under the following assumptions.

{\bf \flushleft Assumptions.} Fix $1 \leq q < p$. From now on, let us assume that $(\mathbb{X},d,\nu)$ is an RCD space, i.e. its Riemannian Ricci curvature is bounded from below. In the case $q = 1$, we additionally require that $\mathbb{X}$ is a proper space. These assumptions are added to remove the dependence of the first-order differential structure on the exponent (see Section \ref{subsec:diffstructure}). Furthermore, for simplicity we will assume that $\nu(\mathbb{X}) < \infty$.

\begin{lemma}
We have $\mathrm{PCM}_p \subset \mathrm{PCM}_q$.
\end{lemma}

\begin{proof}
Let $\{ (f_i, A_i) \}_{i \in \N}$ be an element of the pre-cotangent module $\mathrm{PCM}_p$. In other words, $A_i$ are Borel sets which form a partition of $\mathbb{X}$, $f_i \in D^{1,p}(A_i)$ and
\begin{equation}
\sum_{i \in \N} \int_{A_i} |Df_i|^p \, d\nu < \infty.
\end{equation}
Let us see that $\{ (f_i, A_i) \}_{i \in \N} \in \mathrm{PCM}_q$. We compute
\begin{equation}
\infty > \sum_{i \in \N} \int_{A_i} |Df_i|^p \, d\nu \geq \sum_{i \in \N} \int_{A_i \cap \{ |Df_i| > 1 \} } |Df_i|^p \, d\nu \geq \sum_{i \in \N} \int_{A_i \cap \{ |Df_i| > 1 \} } |Df_i|^q \, d\nu.
\end{equation}
Thus, it follows from $\nu(\mathbb{X}) < \infty$ that
\begin{equation}
\sum_{i \in \N} \int_{A_i} |Df_i|^q \, d\nu = \sum_{i \in \N} \int_{A_i \cap \{ |Df_i| > 1 \} } |Df_i|^q \, d\nu + \sum_{i \in \N} \int_{A_i \cap \{ |Df_i| \leq 1 \} } |Df_i|^q \, d\nu < \infty,
\end{equation}
which proves the claim. By the H\"older inequality this inclusion is continuous.
\end{proof}

\begin{corollary}
We have $\sfrac{\mathrm{PCM}_p}{\sim} \subset \sfrac{\mathrm{PCM}_q}{\sim}.$
\end{corollary}

\begin{proof}
This immediately follows from the previous result and the definition of the equivalence relation, which only identifies sequences $\{ (f_i, A_i) \}$ such that $|Df_i|$ agree $\nu$-a.e. (after possibly taking a more detailed partition).
\end{proof}

\begin{proposition}
We have $L^p(T^* \mathbb{X}) \subset L^q(T^* \mathbb{X})$ and this embedding is continuous.
\end{proposition}

\begin{proof}
We already have the inclusion at the level of pre-cotangent modules and after identifying some elements by the equivalence relation $\sim$. We need to show that every element from $L^p(T^* \mathbb{X})$, which is in general a limit of elements of $\sfrac{\mathrm{PCM}_p}{\sim}$, can be represented as an element of $L^q(T^* \mathbb{X})$ and this map is injective. This follows from the following claim: a Cauchy sequence in $\sfrac{\mathrm{PCM}_p}{\sim}$ is a Cauchy sequence in $\sfrac{\mathrm{PCM}_q}{\sim}$ with the respective norms.

We now proceed to prove the claim. Suppose that $\{ (f_i, A_i) \}_{i \in \N}$ is a Cauchy sequence in $\sfrac{\mathrm{PCM}_p}{\sim}$, there exists $\varepsilon_{n_0} > 0$ such that i.e. for $i,j \geq n_0$
\begin{equation}
\| (f_i,A_i) - (g_j,A_j) \|_{L^p(T^* \mathbb{X})} = \| (f_i - f_j, A_i \cap A_j) \|_{L^p(T^* \mathbb{X})} < \varepsilon_{n_0},
\end{equation}
or equivalently
\begin{equation}
\sum_{i,j \in \N} \int_{A_i \cap A_j} |D(f_i - f_j)|^p \, d\nu < \varepsilon_{n_0}^p.
\end{equation}
Now, denote $B_{ij} = \{ |D(f_i - f_j)| < \varepsilon_{n_0} \}$ and $C_{ij} = \{ |D(f_i - f_j)| \geq \varepsilon_{n_0} \}$. From the previous equation, we conclude that
\begin{align}
\varepsilon_{n_0}^p > \sum_{i,j \in \N} \int_{A_i \cap A_j} |D(f_i - f_j)|^p \, d\nu &\geq \sum_{i,j \in \N} \int_{A_i \cap A_j \cap C_{ij}} |D(f_i - f_j)|^p \, d\nu \\
&\geq
\sum_{i,j \in \N} \int_{A_i \cap A_j \cap C_{ij}} \varepsilon_{n_0}^{p-q} |D(f_i - f_j)|^q \, d\nu.
\end{align}
Therefore,
\begin{align}
\sum_{i,j \in \N} &\int_{A_i \cap A_j} |D(f_i - f_j)|^q \, d\nu \\
&= \sum_{i,j \in \N} \int_{A_i \cap A_j \cap B_{ij}} |D(f_i - f_j)|^q \, d\nu + \sum_{i,j \in \N} \int_{A_i \cap A_j \cap C_{ij}} |D(f_i - f_j)|^q \, d\nu \\
&\leq \varepsilon_{n_0}^q \cdot \nu(\mathbb{X}) + \varepsilon_{n_0}^q,
\end{align}
so
\begin{equation}\label{eq:cauchysequence}
\| (f_i,A_i) - (g_j,A_j) \|_{L^q(T^* \mathbb{X})} = \| (f_i - f_j, A_i \cap A_j) \|_{L^q(T^* \mathbb{X})} < C(\mathbb{X}) \varepsilon_{n_0}
\end{equation}
and $\{ (f_i, A_i) \}_{i \in \N}$ is a Cauchy sequence in $L^q(T^* \mathbb{X})$, so $L^p(T^* \mathbb{X}) \subset L^q(T^* \mathbb{X})$. Moreover, since on the right-hand side of \eqref{eq:cauchysequence} we have $\varepsilon_{n_0}$ multiplied by a universal constant, this embedding is continuous, which concludes the proof.
\end{proof}

\begin{corollary}
We have $L^{q'}(T\mathbb{X}) \subset L^{p'}(T \mathbb{X})$.
\end{corollary}

We are now ready to identify the subdifferential of $\mathcal{F}_{(q,p)}$. Note that the structure of the subdifferential is qualitatively different when $q > 1$ and when $q = 1$; however, even if $q = 1$, it is much more regular than the subdifferential of the total variation \cite{GM2021}, and we do not require the doubling or Poincar\'e assumptions on the metric measure space.

\begin{definition}
Let $q > 1$. We say that $(u,v) \in \mathcal{A}_{(q,p)}$ if and only if $u,v \in L^2(\mathbb{X},\nu)$, $u \in W^{1,p}(\mathbb{X},d,\nu)$ and there exist vector fields $X_1 \in \mathcal{D}^{p',2}(\mathbb{X})$ and $X_2 \in \mathcal{D}^{q',2}(\mathbb{X})$ such that the following conditions hold:
\begin{equation}
-\mathrm{div}(X_1 + X_2) = v \quad \mbox{in } \mathbb{X};
\end{equation}
\begin{equation}
du(X_1) = |du|_*^p = |X_1|^{p'} \quad \nu-\mbox{a.e. in } \mathbb{X};
\end{equation}
\begin{equation}
du(X_2) = |du|_*^q = |X_2|^{q'} \quad \nu-\mbox{a.e. in } \mathbb{X}.
\end{equation}
\end{definition}

We will characterise the subdifferential of $\mathcal{F}_{(q,p)}$ using the operator $\mathcal{A}_{(q,p)}$ by applying the general result from the previous Section. To this end, we set $E_{(q,p)}: L^p(T^* \mathbb{X}) \rightarrow \mathbb{R}$ by the formula
\begin{equation}
E_{(q,p)}(v) = \int_{\mathbb{X}} \bigg( \frac1p |v|_*^p + \frac1q |v|_*^q \bigg) \, d\nu.
\end{equation}
We see that $E_{(q,p)}$ is convex, continuous, nonnegative and coercive, and we have
\begin{equation}
\mathcal{F}_{(q,p)}(u) = \twopartdef{E_{(q,p)}(du)}{\mbox{if } u \in W^{1,p}(\mathbb{X},d,\nu);}{+ \infty}{\mbox{otherwise.}}.
\end{equation}
Therefore, in order to apply Theorem \ref{thm:generalresult}, we only need to characterise the subdifferential of $E_{(q,p)}$ in the duality between $L^p(T^*\mathbb{X})$ and $L^{p'}(T\mathbb{X})$.

\begin{lemma}\label{lem:qpsubdifferential}
For $v \in L^p(T^* \mathbb{X})$, we have
\begin{align}
\partial E_{(q,p)}(v) = \bigg\{ X_1 &+ X_2: \,\, X_1 \in L^{p'}(T\mathbb{X}), \, X_2 \in L^{q'}(T\mathbb{X}),  \\
&v(X_1) = |X_1|^{p'} = |v|_*^p \,\,\, \nu-\mbox{\rm a.e.}, \quad v(X_2) = |X_2|^{q'} = |v|_*^q \,\,\, \nu-\mbox{\rm a.e.} \bigg\}. \nonumber
\end{align}
\end{lemma}

\begin{proof}
Observe that we can decompose the functional $E_{(q,p)}$ into two parts in the following way: $E_{(q,p)} = E_1 + E_2$, where $E_1, E_2: L^p(T^*\mathbb{X}) \rightarrow \mathbb{R}$ are given by
\begin{equation}
E_1(v) = \frac1p \int_{\mathbb{X}} |v|_*^p \, d\nu; \qquad E_2(v) = \frac1q \int_{\mathbb{X}} |v|_*^q \, d\nu.
\end{equation}
Since $E_1$ and $E_2$ are continuous, we have that $\partial E_{(q,p)}$ is the algebraic sum of $\partial E_1$ and $\partial E_2$. We make a computation similar to the one in Example \ref{ex:plaplace} and see that $X_1 \in \partial E_1(v)$ if and only if $v(X_1) = |X_1|^{p'} = |v|_*^p$ $\nu$-a.e in $\mathbb{X}$; analogously, and the condition $X_2 \in \partial E_2(v)$ is equivalent to $v(X_2) = |X_2|^{q'} = |v|_*^q$ $\nu$-a.e. in $\mathbb{X}$.
\end{proof}

Now, let us recall that a multivalued operator $\mathcal{A}$ is completely accretive if and only if the following condition is satisfied (see \cite{ACMBook,BCr2}):
\begin{equation}\label{eq:conditionforaccretivity}
\int_\Omega T(u^1 - u^2) (v^1 - v^2) \, dx \geq 0    
\end{equation}
for every $(u^1,v^1), (u^2,v^2) \in \mathcal{A}$ and all functions $T \in C^{\infty}(\mathbb{R})$ such that $0 \leq T' \leq 1$, $T'$ has compact support, and $x=0$ is not contained in the support of $T$.

\begin{theorem}\label{thm:qpsubdifferential}
We have $\mathcal{A}_{(q,p)} = \partial \mathcal{F}_{(q,p)}$. Furthermore, the operator $\mathcal{A}_{(q,p)}$ is completely accretive and its domain is dense in $L^2(\mathbb{X},\nu)$.
\end{theorem}

\begin{proof}
Since $\mathsf{Ch}_p$ and $\mathsf{Ch}_q$ are convex and lower semicontinuous (see \cite{AGS1}), their sum $\mathcal{F}_{(q,p)}$ also is convex and lower semicontinuous. Furthermore, the functional $E_{(q,p)}$ is convex, continuous, nonnegative and coercive, so we may apply Theorem \ref{thm:generalresult} to conclude that the subdifferential $\partial \mathcal{F}_{(q,p)}$ has dense domain in $L^2(\mathbb{X},\nu)$ and that
\begin{align*}
\partial \mathcal{F}_{(q,p)} = \bigg\{ (u,v) \in L^2(\mathbb{X},\nu): \,\, &u \in W^{1,p}(\mathbb{X},d,\nu) \,\, \mbox{\rm and there exists } X \in \mathcal{D}^{p',2}(\mathbb{X}) \\
& \quad \mbox{\rm s.t.} \,\, -\mathrm{div}(X) = v \,\, \mbox{in } \mathbb{X} \quad \mbox{\rm and} \quad X \in \partial E_{(q,p)}(du) \bigg\}.
\end{align*}
The conclusion that $\mathcal{A}_{(q,p)} = \partial \mathcal{F}_{(q,p)}$ immediately follows from the characterisation of the subdifferential of $E_{(q,p)}$ given in Lemma \ref{lem:qpsubdifferential}. 

Therefore, the only claim to be proved is complete accretivity of $\mathcal{A}_{(q,p)}$. We need to check that condition \eqref{eq:conditionforaccretivity} is satisfied. For $i = 1,2$, let $(u^i, v^i) \in \mathcal{A}_{(q,p)}$, and let $X_1^i, X_2^i$ be the associated vector fields. Since $T(u^1 - u^2) \in W^{1,p}(\mathbb{X},d,\nu)$, we use the chain rule (Proposition \ref{prop:chainrule}) to get
\begin{align*}
\int_{\mathbb{X}} T &(u^1-u^2)(v^1-v^2)\, d\nu = -  \int_{\mathbb{X}} T(u^1-u^2) (\mathrm{div}(X_1^1 + X_2^1) - \mathrm{div}(X_1^2 + X_2^2)) \, d\nu \\
&= -  \int_{\mathbb{X}} T(u^1-u^2) \, \mathrm{div}(X_1^1 - X_1^2) \, d\nu -  \int_{\mathbb{X}} T(u^1-u^2) \, \mathrm{div}(X_2^1 - X_2^2) \, d\nu \\
&= \int_{\mathbb{X}} dT(u^1-u^2)(X_1^1 - X_1^2) \, d \nu + \int_{\mathbb{X}} dT(u^1-u^2)(X_2^1 - X_2^2) \, d \nu \\
&= \int_{\mathbb{X}} T'(u^1-u^2) \, d(u^1-u^2) (X_1^1 - X_1^2) \, d \nu + \int_{\mathbb{X}} T'(u^1-u^2) \, d(u^1-u^2) (X_2^1 - X_2^2) \, d \nu.
\end{align*}
We will estimate the two terms separately. For the first term, we calculate
\begin{align*}
\int_{\mathbb{X}} &T'(u^1-u^2) \, d(u^1-u^2) (X_1^1 - X_1^2) \, d \nu \\
& =\int_{\mathbb{X}} T'(u^1-u^2) \, du^1(X_1^1) \, d\nu -  \int_{\mathbb{X}} T'(u^1-u^2) \, du^1(X_1^2) \, d\nu \\
&\qquad\qquad - \int_{\mathbb{X}} T'(u^1-u^2) \, du^2(X_1^1) \, d\nu +\int_{\mathbb{X}} T'(u^1-u^2) \, du^2(X_1^2) \, d\nu \\
&= \int_{\mathbb{X}} T'(u^1-u^2)  |du^1|_*^p \, d\nu -  \int_{\mathbb{X}} T'(u^1-u^2) \, du^1(X_1^2) \, d\nu \\
&\qquad\qquad - \int_{\mathbb{X}} T'(u^1-u^2) \, du^2(X_1^1) \, d\nu + \int_{\mathbb{X}} T'(u^1-u^2) |du^2|_*^p \, d\nu
\end{align*}
\begin{align*}
&\geq \int_{\mathbb{X}} T'(u^1-u^2)  |du^1|_*^p \, d \nu - \frac1p \int_{\mathbb{X}} T'(u^1-u^2)  |du^1|_*^p \, d\nu \\
&\qquad\qquad - \frac{1}{p'} \int_{\mathbb{X}} T'(u^1-u^2) \vert X_1^2 \vert^{p'} \, d \nu - \frac1p \int_{\mathbb{X}} T'(u^1-u^2)  |du^2|_*^p \, d\nu \\
&\qquad\qquad - \frac{1}{p'} \int_{\mathbb{X}} T'(u^1-u^2) \vert X_1^1 \vert^{p'} \, d \nu + \int_{\mathbb{X}} T'(u^1-u^2) |du^2|_*^p \, d \nu \\
&= \frac{1}{p'} \int_{\mathbb{X}} T'(u^1-u^2)  |du^1|_*^p \, d \nu - \frac{1}{p'} \int_{\mathbb{X}} T'(u^1-u^2) \vert X_1^1\vert^{p'} \, d \nu \\
&\qquad\qquad + \frac{1}{p'} \int_{\mathbb{X}} T'(u^1-u^2)  |du^2|_*^p \, d \nu - \frac{1}{p'} \int_{\mathbb{X}} T'(u^1-u^2) \vert X_1^2\vert^{p'} \, d \nu = 0,
\end{align*}
since $T' \geq 0$ and $| X_1^i |^{p'} =| du^i |_*^p$ $\nu$-a.e. for $i=1,2$. An identical argument shows that
\begin{equation*}
\int_{\mathbb{X}} T'(u^1-u^2) \, d(u^1-u^2) (X_2^1 - X_2^2) \, d \nu \geq 0,
\end{equation*}
so the operator $\mathcal{A}_{(q,p)}$ is completely accretive.
\end{proof}

Observe that the above proof of complete accretivity cannot be replicated (and, in fact, is not expected to hold) in the general case studied in Section \ref{sec:general}. This is due to the fact that the condition $X \in \partial E(du)$ does not in general have a pointwise interpretation, and therefore we cannot use it inside an integral which involves a weight $T'(u^1 - u^2)$. In the above statement, the argument worked since the subdifferential of $\partial\mathcal{F}_{(q,p)}$ has a pointwise characterisation.

Now, we will show that a similar result holds in the limiting case $q = 1$.

\begin{definition}
We say that $(u,v) \in \mathcal{A}_{(1,p)}$ if and only if $u,v \in L^2(\mathbb{X},\nu)$, $u \in W^{1,p}(\mathbb{X},d,\nu)$ and there exist vector fields $X_1 \in \mathcal{D}^{p',2}(\mathbb{X})$ and $X_2 \in \mathcal{D}^{\infty,2}(\mathbb{X})$ such that the following conditions hold:
\begin{equation}
-\mathrm{div}(X_1 + X_2) = v \quad \mbox{in } \mathbb{X};
\end{equation}
\begin{equation}
du(X_1) = |du|_*^p = |X_1|^{p'} \quad \nu-\mbox{a.e. in } \mathbb{X};
\end{equation}
\begin{equation}
\| |X_2| \|_{L^\infty(\mathbb{X},\nu)} \leq 1; \qquad du(X_2) = |du|_* \quad \nu-\mbox{a.e. in } \mathbb{X}.
\end{equation}
\end{definition}

Again, we will characterise the subdifferential of $\mathcal{F}_{(1,p)}$ using the operator $\mathcal{A}_{(1,p)}$ by applying the general result from the previous Section (Theorem \ref{thm:generalresult}). To this end, we set $E_{(1,p)}: L^p(T^* \mathbb{X}) \rightarrow \mathbb{R}$ by the formula
\begin{equation}
E_{(1,p)}(v) = \int_{\mathbb{X}} \bigg( \frac1p |v|_*^p + |v|_* \bigg) \, d\nu,
\end{equation}
so that
\begin{equation}
\mathcal{F}_{(1,p)}(u) = \twopartdef{E_{(1,p)}(du)}{\mbox{if } u \in W^{1,p}(\mathbb{X},d,\nu);}{+ \infty}{\mbox{otherwise},}
\end{equation}
and characterise its subdifferential in the duality between $L^p(T^*\mathbb{X})$ and $L^{p'}(T\mathbb{X})$.

\begin{lemma}\label{lem:1psubdifferential}
For $v \in L^p(T^* \mathbb{X})$, we have
\begin{align*}
\partial E_{(1,p)}(v) = \bigg\{ &X_1 + X_2: \,\, X_1 \in L^{p'}(T\mathbb{X}), \, X_2 \in L^{\infty}(T\mathbb{X}),  \\
&v(X_1) = |X_1|^{p'} = |v|_*^p \,\,\, \nu-\mbox{\rm a.e}, \quad \| |X_2| \|_{L^\infty(\mathbb{X},\nu)} \leq 1, \quad v(X_2) = |v|_* \,\,\, \nu-\mbox{\rm a.e.} \bigg\}. \nonumber
\end{align*}
\end{lemma}

\begin{proof}
We decompose the functional $E_{(1,p)}$ in the following way: $E_{(1,p)} = E_1 + E_2$, where $E_1, E_2: L^p(T^*\mathbb{X}) \rightarrow \mathbb{R}$ are given by
\begin{equation}
E_1(v) = \frac1p \int_{\mathbb{X}} |v|_*^p \, d\nu; \qquad E_2(v) = \int_{\mathbb{X}} |v|_* \, d\nu.
\end{equation}
Since $E_1$ and $E_2$ are continuous, we have that $\partial E_{(1,p)}$ is the algebraic sum of $\partial E_1$ and $\partial E_2$. We make a computation similar to the one in Example \ref{ex:plaplace} and see that $X_1 \in \partial E_1(v)$ if and only if $v(X_1) = |X_1|^{p'} = |v|_*^p$ $\nu$-a.e in $\mathbb{X}$. For the second subdifferential, observe that the convex conjugate $E_2^*: L^\infty(T\mathbb{X}) \rightarrow \mathbb{R}$ is given by
\begin{equation}
E_2(v^*) = \twopartdef{0}{\mbox{if } \| |v^*| \|_{L^\infty(\mathbb{X},\nu)} \leq 1;}{+\infty}{\mbox{otherwise.}}
\end{equation}
Hence, since the condition $X_2 \in \partial E_2(v)$ is equivalent to
\begin{equation}
E_2(v) + E_2^*(X_2) = \int_{\mathbb{X}} v(X_2) \, d\nu,
\end{equation}
we conclude that $X_2 \in \partial E_2(v)$ if and only if $\| |X_2| \|_{L^\infty(\mathbb{X},\nu)} \leq 1$ and
\begin{equation}
\int_{\mathbb{X}} |v|_* \, d\nu = \int_{\mathbb{X}} v(X_2) \, d\nu,
\end{equation}
or equivalently $\| |X_2| \|_{L^\infty(\mathbb{X},\nu)} \leq 1$ and $v(X_2) = |v|_*$ $\nu$-a.e. in $\mathbb{X}$.
\end{proof}

\begin{theorem}\label{thm:1psubdifferential}
We have $\mathcal{A}_{(1,p)} = \partial \mathcal{F}_{(1,p)}$. Furthermore, the operator $\mathcal{A}_{(1,p)}$ is completely accretive and its domain is dense in $L^2(\mathbb{X},\nu)$.
\end{theorem}

\begin{proof}
Since $\mathsf{Ch}_p$ and $\mathsf{Ch}_1$ are convex and lower semicontinuous (see \cite{ADiM,AGS1}), their sum $\mathcal{F}_{(1,p)}$ also is convex and lower semicontinuous. Furthermore, the functional $E_{(1,p)}$ is convex, continuous, nonnegative and coercive, so we may apply Theorem \ref{thm:generalresult} to conclude that the subdifferential $\partial \mathcal{F}_{(1,p)}$ has dense domain in $L^2(\mathbb{X},\nu)$ and that
\begin{align*}
\partial \mathcal{F}_{(1,p)} = \bigg\{ (u,v) \in L^2(\mathbb{X},\nu): \,\, &u \in W^{1,p}(\mathbb{X},d,\nu) \,\, \mbox{\rm and there exists } X \in \mathcal{D}^{p',2}(\mathbb{X}) \\
& \quad \mbox{\rm s.t.} \,\, -\mathrm{div}(X) = v \,\, \mbox{in } \mathbb{X} \quad \mbox{\rm and} \quad X \in \partial E_{(1,p)}(du) \bigg\}.
\end{align*}
The conclusion that $\mathcal{A}_{(1,p)} = \partial \mathcal{F}_{(1,p)}$ immediately follows from the characterisation of the subdifferential of $E_{(1,p)}$ given in Lemma \ref{lem:1psubdifferential}. 

Therefore, the only claim to be proved is complete accretivity of $\mathcal{A}_{(1,p)}$. Take $T: \mathbb{R} \rightarrow \R$ to be a function as in condition \eqref{eq:conditionforaccretivity}; we will show that the condition is satisfied. For $i = 1,2$, let $(u^i, v^i) \in \mathcal{A}_{(1,p)}$, and let $X_1^i, X_2^i$ be the associated vector fields. We proceed similarly to the proof of Theorem \ref{thm:qpsubdifferential} and get that
\begin{align*}
\int_{\mathbb{X}} T &(u^1-u^2)(v^1-v^2)\, d\nu \\
&= \int_{\mathbb{X}} T'(u^1-u^2) \, d(u^1-u^2) (X_1^1 - X_1^2) \, d \nu + \int_{\mathbb{X}} T'(u^1-u^2) \, d(u^1-u^2) (X_2^1 - X_2^2) \, d \nu.
\end{align*}
Working as in the proof of Theorem \ref{thm:qpsubdifferential}, we get that the first term is nonegative. We now estimate the second term:
\begin{align*}
\int_{\mathbb{X}} &T'(u^1-u^2) \, d(u^1-u^2) (X_2^1 - X_2^2) \, d \nu \\
& =\int_{\mathbb{X}} T'(u^1-u^2) \, du^1(X_2^1) \, d\nu -  \int_{\mathbb{X}} T'(u^1-u^2) \, du^1(X_2^2) \, d\nu \\
&\qquad\qquad - \int_{\mathbb{X}} T'(u^1-u^2) \, du^2(X_2^1) \, d\nu +\int_{\mathbb{X}} T'(u^1-u^2) \, du^2(X_2^2) \, d\nu \\
&= \int_{\mathbb{X}} T'(u^1-u^2)  |du^1|_* \, d\nu -  \int_{\mathbb{X}} T'(u^1-u^2) \, du^1(X_2^2) \, d\nu \\
&\qquad\qquad - \int_{\mathbb{X}} T'(u^1-u^2) \, du^2(X_2^1) \, d\nu + \int_{\mathbb{X}} T'(u^1-u^2) |du^2|_* \, d\nu \\
&\geq \int_{\mathbb{X}} T'(u^1-u^2)  |du^1|_* \, d \nu - \int_{\mathbb{X}} T'(u^1-u^2)  |du^1|_* \, d\nu \\
&\qquad\qquad - \int_{\mathbb{X}} T'(u^1-u^2)  |du^2|_* \, d\nu + \int_{\mathbb{X}} T'(u^1-u^2) |du^2|_* \, d \nu = 0,
\end{align*}
since $T' \geq 0$, $\| |X_2| \|_{L^\infty(\mathbb{X},\nu)} \leq 1$ and $du^i(X_2^i) =| du^i |_*$ $\nu$-a.e. for $i=1,2$.
\end{proof}

In light of the above results, for any $q \in [1,\infty)$ it is natural to introduce the following concept of solution to the gradient flow of the functional $\mathcal{F}_{(q,p)}$.

\begin{definition}
We define in $L^2(\mathbb{X},\nu)$ the multivalued operator $\Delta_{(q,p)}$ by
\begin{center}$(u, v) \in \Delta_{(q,p)}$ \ if and only if \ $-v \in \partial \mathcal{F}_{(q,p)}(u)$.
\end{center}
\end{definition}

In this case, the abstract Cauchy problem \eqref{eq:abstractcauchy} corresponds to the Cauchy problem for the $(q,p)$-Laplacian, i.e.,
\begin{equation}\label{eq:qplaplacian}
\left\{ \begin{array}{ll} \frac{d}{dt} u(t) \in  \Delta_{(q,p)}(u(t)), \quad t \in [0,T] \\[5pt] u(0) = u_0. \end{array}\right.
\end{equation}

\begin{definition}\label{dfn:qplaplaceflow}
{\rm Given $u_0 \in L^2(\mathbb{X},\nu)$, we say that $u$ is a {\it weak solution} of the Cauchy problem \eqref{eq:qplaplacian}, if { $u \in C([0,T];L^2(\mathbb{X},\nu)) \cap W_{\rm loc}^{1,2}(0, T; L^2(\mathbb{X},\nu))$}, $u(0, \cdot) = u_0$, and for almost all $t \in (0,T)$
\begin{equation}
u_t(t, \cdot) \in \Delta_{(q,p)} u(t, \cdot).
\end{equation}
In other words, for almost all $t \in (0,T)$ we have $u(t) \in W^{1,p}(\mathbb{X}, d, \nu)$ and:

{\flushleft ($q > 1$)} there exist vector fields $X_1(t) \in  \mathcal{D}^{p',2}(\mathbb{X})$ and $X_2(t) \in \mathcal{D}^{q',2}(\mathbb{X})$ such that for almost all $t \in (0,T)$ the following conditions hold:
$$ u_t(t, \cdot) = \mathrm{div}(X_1(t) + X_2(t)) \quad \hbox{in } \mathbb{X}; $$
$$ | X_1(t) |^{p'} = du(t)(X_1(t)) = |du(t)|_*^p \quad \nu\hbox{-a.e. in } \mathbb{X};$$
$$ | X_2(t) |^{q'} = du(t)(X(t)) = |du(t)|_*^q \quad \nu\hbox{-a.e. in } \mathbb{X};$$

{\flushleft ($q = 1$)} there exist vector fields $X_1(t) \in  \mathcal{D}^{p',2}(\mathbb{X})$ and $X_2(t) \in \mathcal{D}^{\infty,2}(\mathbb{X})$ such that for almost all $t \in (0,T)$ the following conditions hold:
$$ u_t(t, \cdot) = \mathrm{div}(X_1(t) + X_2(t)) \quad \hbox{in } \mathbb{X}; $$
$$ | X_1(t) |^{p'} = du(t)(X_1(t)) = |du(t)|_*^p \quad \nu\hbox{-a.e. in } \mathbb{X};$$
$$ \| | X_2(t) | \|_{L^\infty(\mathbb{X},\nu)} \leq 1; \quad  du(t)(X(t)) = |du(t)|_* \quad \nu\hbox{-a.e. in } \mathbb{X}.$$
}
\end{definition}

Following the results in \cite{LP,LPR}, is some particular settings such as weighted Euclidean spaces or Riemannian manifolds one may truly interpret the above definition in a pointwise sense; see \cite{GM2021} for a similar analysis for the $p$-Laplacian evolution equation.

As a consequence of the characterisation given in Theorems \ref{thm:qpsubdifferential} and \ref{thm:1psubdifferential}, by the Brezis-Komura theorem (Theorem \ref{thm:breziskomura}), for all $q \in [1,p)$ we have the following existence and uniqueness result. The comparison principle is a consequence of the complete accretivity of the operator $\mathcal{A}_{(q,p)}$.

\begin{theorem}\label{ExisUniqp}
For any $u_0 \in L^2(\mathbb{X}, \nu)$ and all $T > 0$ there exists a unique weak solution $u(t)$ of the Cauchy problem \eqref{eq:qplaplacian} with $u(0) =u_0$.  Moreover, the following comparison principle holds: for all $r \in [1,\infty]$, if $u_1, u_2$ are weak solutions for the initial data $u_{1,0}, u_{2,0} \in  L^2(\mathbb{X}, \nu) \cap  L^r(\mathbb{X}, \nu)$ respectively, then
\begin{equation}\label{CompPrincipleplaplace}
\Vert (u_1(t) - u_2(t))^+ \Vert_r \leq \Vert ( u_{1,0}- u_{2,0})^+ \Vert_r.
\end{equation}
\end{theorem}

Note that while complete accretivity of the subdifferential of the $p$-Cheeger energy was known and could be proved using other methods \cite{Kell} (see also \cite{AGS2} for the case $p = 2$), it is not immediate that the subdifferential of the sum of Cheeger energies with different exponents is completely accretive, since the condition \eqref{eq:conditionforaccretivity} is nonlinear. The method presented in this paper provides an effective way of proving complete accretivity for functionals for which we can find a pointwise representation of the subdifferential $\partial E$.

\noindent {\bf Acknowledgments.} This research was funded partially by the Austrian Science Fund (FWF), grant ESP 88. For the purpose of open access, the author has applied a CC BY public copyright licence to any Author Accepted Manuscript version arising from this submission.

\end{document}